\documentclass{amsart}

\usepackage[utf8]{inputenc}
\usepackage[utf8]{inputenc}   
\usepackage[T1]{fontenc}      
\usepackage{lmodern}

\usepackage{amssymb,amsmath,amsfonts,latexsym}
\usepackage{graphicx}
\usepackage[all]{xy}
\usepackage{kmath}
\usepackage{cite}
\usepackage[T1]{fontenc}
\usepackage{hyperref}
\usepackage[parfill]{parskip}
\usepackage{float}

\usepackage{tikz}
\usepackage{newtxtext,newtxmath}
\usepackage{tkz-euclide}
\usetikzlibrary{decorations.markings,arrows}
\tikzset{
    arrowMe/.style={
        postaction=decorate,
        decoration={
            markings,
            mark=at position .5 with {\arrow[thick]{#1}}
        }
    }
}

\usetikzlibrary{backgrounds,calc}
\tikzset{point/.style = {fill=black,circle,inner sep=0.7pt}}

\pgfarrowsdeclaredouble{<<s}{>>s}{stealth}{stealth}
\pgfarrowsdeclaretriple{<<<s}{>>>s}{stealth}{stealth}

\tikzstyle{vertex}=[circle,fill=black!25,minimum size=20pt,inner sep=0pt]
\tikzstyle{edge} = [draw]
\tikzstyle{weight} = [font=\tiny]

\title{On the parity of the Betti numbers of 3-manifolds with a parallel vector field}

\author{Emmanuel Gnandi\and Raymond A. Hounnonkpe}

\address{
University of Toulouse, France\\
and\\
Institut de Mathématiques et de Sciences Physiques (IMSP),\\
Université d'Abomey-Calavi, Porto-Novo, Bénin.
}

\email{kpanteemmanuel@gmail.com, rhounnonkpe@ymail.com}

\newtheorem{thm}{Theorem}[section]
\newtheorem{cor}{Corollary}[section]

\newtheorem{defn}[thm]{Definition}

\numberwithin{equation}{section}

\begin{document}
\maketitle
\begin{abstract}
The question of whether a closed, orientable manifold can admit a nontrivial vector field that is parallel with respect to some Riemannian metric is a classical problem in differential geometry, first posed by S. S. Chern \cite{chern1966geometry}. In this work, we provide a complete answer to Chern's question in dimension three. Specifically, we show that a closed, orientable 3-manifold admits a nontrivial parallel vector field with respect to some Riemannian metric if and only if it is a Kähler mapping torus. Furthermore, we prove that the Betti numbers of any such 3-manifold are necessarily odd. A full classification of these manifolds is also obtained. Similar results are established for compact orientable Lorentzian 3-manifolds admitting either a parallel timelike vector field or a parallel lightlike vector field.
\end{abstract}
\section{Introduction}

In 1966, S. S. Chern~\cite{chern1966geometry} posed the problem of determining when a compact orientable manifold $M$ admits a nontrivial vector field that is parallel with respect to some Riemannian metric (\cite{chern1966geometry} Example 3). He observed that the first two Betti numbers necessarily satisfy
\[
b_1(M) \ge 1, \qquad b_2(M) \ge b_1(M) - 1,
\]
a condition now referred to as Chern’s criterion, and conjectured that these inequalities are not sufficient~\cite{chern1966geometry}.

The situation for manifolds with boundary was clarified by P. Percell~\cite{percell1981parallel}, who established that every compact manifold with boundary admits a vector field parallel with respect to some Riemannian metric. For closed manifolds, however, the problem remains considerably more intricate: identifying the precise topological conditions that are both necessary and sufficient for the existence of a nonvanishing parallel vector field is still a subtle challenge. It is known~\cite{karp1977parallel} that if a closed $n$-dimenional Riemannian manifold carries a parallel vector field, then its Euler–Poincaré characteristic must vanish. Furthermore, Bott~\cite{bott1967vector} demonstrated that all Pontryagin numbers of such a manifold are zero. Extending these observations, Karp~\cite{karp1977parallel} derived additional necessary conditions on the Betti numbers, namely
\[
b_1(M) \ge 1, \qquad 
b_{l+1}(M) \ge b_l(M) - b_{l-1}(M), \quad \text{for all } 1 \le l \le n-1.
\]
He also constructed an explicit example of a manifold satisfying both Chern’s criterion and Bott’s vanishing conditions on the Pontryagin numbers, yet admitting no parallel vector field for any Riemannian metric thus confirming Chern’s conjecture that the criterion is not sufficient.

A complete topological characterization was later obtained by Welsh~\cite{welsh1986manifolds}, who showed that compact manifolds admitting a parallel vector field are precisely compact fiber bundles over tori with finite structural group.

The purpose of this note is to resolve Chern's question in dimension $3$. We prove that the only $3$-dimensional closed orientable manifolds that admit a nontrivial parallel vector field with respect to a Riemannian metric are K\"ahler mapping torus. As a consequence, we show that their Betti numbers are necessarily odd, and we provide a full classification. Analogous results are derived in the Lorentzian framework, where we consider $3$-manifolds endowed with either parallel timelike or lightlike vector fields.

\section{Basic Notions}

\subsection{Kähler Mapping Torus}

We begin by recalling one of the fundamental constructions underlying our results.

\begin{defn}
Let $\varphi \in \mathrm{Diff}(\Sigma)$ be a diffeomorphism of a closed, connected manifold $\Sigma$.  
The mapping torus of $\varphi$ is defined as
\[
\Sigma_{\varphi} = \frac{\Sigma \times [0,1]}{(x,0) \sim (\varphi(x),1)}.
\]
It forms the total space of a smooth fiber bundle
\[
\Sigma \hookrightarrow \Sigma_{\varphi} \xrightarrow{\;\pi\;} \mathbb{S}^1.
\]
\end{defn}

If $(\Sigma, \Omega)$ is a symplectic manifold and $\varphi$ is a symplectomorphism, i.e. $\varphi^{*}\Omega = \Omega$, then $\Sigma_{\varphi}$ is called a symplectic mapping torus.  
Moreover, if $(\Sigma, \Omega)$ is a Kähler manifold endowed with a Hermitian structure $(J,h)$, and if $\varphi$ is a Hermitian isometry satisfying
\[
\varphi_{*} \circ J = J \circ \varphi_{*}, 
\qquad 
\varphi^{*}h = h,
\]
(which automatically implies $\varphi^{*}\Omega = \Omega$), then $\Sigma_{\varphi}$ is referred to as a Kähler mapping torus \cite{li2008topology}.


\section{Main Results}
\label{thm:Chern}

The main aim of this paper is to
 classify $M$ which carry a nontrivial vector field that is
 parallel with respect to a Riemannian metric.
 
\begin{thm}
Let \( M \) be a closed, orientable \(3\)-manifold. The following statements are equivalent:
\begin{enumerate}
    \item \(M\) admits a closed \(1\)-form \(\alpha\) and a Killing vector field \(Z\) such that
    \[
    \alpha(Z) = \text{constant} \neq 0.
    \]

    \item \(M\) is a K\"ahler mapping torus.

    \item \(M\) admits a nontrivial parallel vector field with respect to some Riemannian metric.
\end{enumerate}
\end{thm}

\begin{proof}
We prove the equivalence by establishing the implications (1) \(\Rightarrow\) (2), (2) \(\Rightarrow\) (3), and (3) \(\Rightarrow\) (1).

\noindent\text{Proof that (1) \(\Rightarrow\) (2).}
Assume that \(M\) admits a closed \(1\)-form \(\alpha\) and a Killing vector field \(Z\) with respect to a Riemannian metric \(g\), such that \(\alpha(Z)\) is a nonzero constant. We may assume without loss of generality that \(\alpha(Z) = 1\). Let \(v_g\) denote the Riemannian volume form associated with \(g\), and define the \(2\)-form
\[
\Omega := \iota_{Z} v_g.
\]
A straightforward computation gives
\[
d\Omega = d(\iota_{Z} v_g) = \mathcal{L}_{Z} v_g = 0,
\]
where the last equality follows from the fact that \(Z\) is a Killing vector field. Moreover, we observe that
\[
\iota_{Z}(\alpha \wedge v_g) = 0,
\]
which implies
\[
v_g = \alpha \wedge \Omega.
\]

From the above, we deduce that \(\alpha(Z) = 1\) and \(\iota_Z \Omega = 0\). Hence, \(Z\) is the Reeb vector field of the cosymplectic structure \((\alpha, \Omega)\). Since \(Z\) is a Killing vector field with respect to \(g\), it follows from \cite{bazzoni2015k} that \((M, \alpha, \Omega, g)\) is a K-cosymplectic manifold. By \cite[Proposition~2.8]{bazzoni2015k}, \(M\) carries an almost contact metric structure \((\eta, \phi, \xi, h)\) with \(\eta = \alpha\) and \(\xi = Z\), satisfying \(\mathcal{L}_{Z}h = 0\). Furthermore, by \cite[Proposition~3]{goldberg1969integrability}, this structure is normal (i.e., the Nijenhuis tensor \(N_\phi\) vanishes), and hence \((M, \alpha, Z, \phi, h)\) is a co-K\"ahler manifold. Finally, by \cite[Theorem~2]{li2008topology}, \(M\) is a K\"ahler mapping torus. This completes the proof of the first implication.

\noindent\text{Proof that (2) \(\Rightarrow\) (3).}
Suppose \(M\) is a K\"ahler mapping torus. By \cite[Theorem~2]{li2008topology}, \(M\) admits a co-K\"ahler structure \((\phi, \xi, \eta, g)\) satisfying
\[
\phi^2 = -I + \eta \otimes \xi, \qquad g(X,Y) = g(\phi X, \phi Y) + \eta(X)\eta(Y).
\]
The metric \(g\) is compatible with the almost contact structure \((\phi, \xi, \eta)\), and we have \(\eta(X) = g(X, \xi)\). Moreover, the Nijenhuis tensor vanishes:
\[
N^\phi(X,Y) = [\phi X, \phi Y] - \phi[\phi X,Y] - \phi[X, \phi Y] + \phi^2[X,Y] = 0.
\]

As established in \cite{blair2010riemannian, blair1966theory, li2008topology, marrero1998new, cappelletti2013survey, olszak1981almost}, every compact co-K\"ahler manifold satisfies
\[
\nabla^{LC} \eta = 0,
\]
where \(\nabla^{LC}\) denotes the Levi-Civita connection associated with \(g\). Since \(\eta(X) = g(X, \xi)\), a straightforward calculation shows that the Reeb vector field \(\xi\) is parallel:
\[
\nabla^{LC} \xi = 0.
\]
This completes the proof of the second implication.

\noindent\text{Proof that (3) \(\Rightarrow\) (1).}
Assume that \(M\) admits a nontrivial parallel vector field \(Z\) with respect to a Riemannian metric \(g\). Let \(\beta := \iota_Z g\) denote the \(1\)-form associated with \(Z\), and let \(D\) be the Levi-Civita connection of \(g\). Since \(D Z = 0\), it follows that \(D\beta = 0\), and consequently \(d\beta = 0\). Furthermore, the function \(g(Z, Z) = \beta(Z)\) is constant and positive on \(M\), as \(Z\) is nontrivial. The vector field $Z$ is a Killing vector field (since any parallel vector field is Killing) and \(\beta(Z)\) is a positive constant. This completes the proof of the third implication and the theorem.
\end{proof}

From Theorem~\ref{thm:Chern}, the only $3$-dimensional closed, orientable manifolds that admit a nontrivial parallel vector field with respect to some Riemannian metric are the Kähler mapping torus. 
This completely answers Chern’s question in dimension~$3$.\\

Yves Carrière~\cite{carriere1984flots} classified Riemannian flows on compact, orientable $3$-manifolds, and deduced that any such manifold admitting a nontrivial Killing flow is a Seifert manifold. 
Now, consider a $3$-manifold $M$ endowed with a nontrivial parallel vector field $Z$ with respect to some Riemannian metric. 
Since $Z$ is necessarily a Killing vector field, it follows from~\cite{carriere1984flots} that $M$ is a Seifert manifold. 
In what follows, we study the parity of the Betti numbers of $M$ and provide a complete classification of $3$-manifolds admitting a nontrivial parallel vector field with respect to some Riemannian metric. 
This result may be viewed as a refinement of Carrière’s classification.

\begin{thm}\label{thm:Betti}
Let \( M \) be a closed, orientable \(3\)-manifold admitting a nontrivial parallel vector field with respect to some Riemannian metric. Then:
\begin{enumerate}
    \item All Betti numbers of \( M \) are odd.
    
    \item \( M \) is diffeomorphic to one of the following:
    \begin{enumerate}
        \item the product \( \mathbb{S}^2 \times \mathbb{S}^1 \);
        \item the \(3\)-torus \( \mathbb{T}^3 \);
        \item a nontrivial torus bundle \( \mathbb{T}^2_A \) over \( \mathbb{S}^1 \), where
        \[
        A \in \left\{
        \begin{pmatrix} 0 & 1 \\ -1 & 0 \end{pmatrix},
        \begin{pmatrix} -1 & 0 \\ 0 & -1 \end{pmatrix},
        \begin{pmatrix} 0 & -1 \\ 1 & 1 \end{pmatrix},
        \begin{pmatrix} -1 & -1 \\ 1 & 0 \end{pmatrix}
        \right\};
        \]
        \item a quotient \( (\mathbb{H}^2 \times \mathbb{R}) / \Gamma \), where \( \Gamma \subset \mathrm{Isom}(\mathbb{H}^2 \times \mathbb{R}) \) is a discrete subgroup acting freely.
    \end{enumerate}
\end{enumerate}
\end{thm}

\begin{proof}
We proceed in several steps.

\noindent\text{(1).}
By Theorem~\ref{thm:Chern}, \( M \) is a Kähler mapping torus. According to \cite[Theorem~2]{li2008topology}, such a manifold is co-Kähler. Furthermore, by \cite[Theorem~11]{Chinea1993TopologyOC}, the first Betti number \( b_1(M) \) is odd. Applying Poincaré duality, it follows that all Betti numbers \( b_k(M) \) are odd.

\noindent\text{(2).}
Since \( M \) is a Kähler mapping torus, we have \( M \simeq \Sigma_\varphi \), where \( (\Sigma, J, h) \) is a Kähler manifold. Let \( \mathrm{Isom}(\Sigma, h) \) denote the group of Hermitian isometries of \( \Sigma \), and \( \mathrm{Isom}_0(\Sigma, h) \) its identity component. By \cite[Theorem~6.6]{bazzoni2014structure}, the diffeomorphism \( \varphi \) is either isotopic to the identity or of finite order in \( \mathrm{Isom}(\Sigma, h) / \mathrm{Isom}_0(\Sigma, h) \); in particular, \( \varphi \) is periodic. By \cite[Theorem~5.4]{scott1983geometries}, \( M \) is a Seifert fibered space over a 2-dimensional orbifold \( B \) with Euler number zero. Moreover, as noted in \cite{scott1983geometries}, the geometries \( \mathbb{S}^3 \), \( \mathrm{Nil} \), and \( \widetilde{\mathrm{SL}_2(\mathbb{R})} \) are excluded. The classification now depends on the Euler characteristic \( \chi(B) \) of the base orbifold.
\noindent\text{Case 1: \( \chi(B) > 0 \).}
Here \( M \) has \( \mathbb{S}^2 \times \mathbb{R} \) geometry. By \cite[Proposition~10.3.36]{martelli2016introduction}, \( M \) is diffeomorphic to one of: \( \mathbb{S}^2 \times \mathbb{S}^1 \), the orientable \( \mathbb{S}^1 \)-bundle over \( \mathbb{RP}^2 \), or \( \bigl(S^2; (p,q), (p,-q)\bigr) \). All manifolds of type \( \bigl(S^2; (p,q), (p,-q)\bigr) \) are diffeomorphic to \( \mathbb{S}^2 \times \mathbb{S}^1 \). By \cite[Exercise~10.3.37]{martelli2016introduction}, the orientable \( \mathbb{S}^1 \)-bundle over \( \mathbb{RP}^2 \) is diffeomorphic to \( \mathbb{RP}^3 \# \mathbb{RP}^3 \), which has \( b_1(M) = 0 \) and is therefore excluded. Hence, the only possibility is \( M \simeq \mathbb{S}^2 \times \mathbb{S}^1 \). \noindent\text{Case 2: \( \chi(B) = 0 \).}
Here \( M \) has Euclidean geometry. By \cite[Theorem~12.3.1]{martelli2016introduction}, \( M \) admits a flat metric. From the classification of closed flat 3-manifolds \cite{hantzsche1935dreidimensionale, Wolf, bieberbach1911bewegungsgruppen,bieberbach1912bewegungsgruppen}, the first homology group \( H_1(M, \mathbb{Z}) \) is isomorphic to one of: \( \mathbb{Z}^3 \), \( \mathbb{Z} \oplus \mathbb{Z}_2^2 \), \( \mathbb{Z} \oplus \mathbb{Z}_3 \), \( \mathbb{Z} \oplus \mathbb{Z}_2 \), \( \mathbb{Z} \), or \( \mathbb{Z}_4^2 \). The Hantzsche--Wendt manifold (with \( H_1(M, \mathbb{Z}) \cong \mathbb{Z}_4^2 \)) has \( b_1(M) = 0 \) and is excluded. All remaining manifolds have positive first Betti number, and by \cite[Theorem~3.2]{marrero1998new}, \( M \) is diffeomorphic to one of: \( M_1(1) \), \( M_2(1) \), \( M'_1(1) \), \( M'_2(1) \), or \( \mathbb{T}^3 \), in the notation of \cite{marrero1998new}. Moreover, as shown in \cite{marrero1998new}, these are precisely the torus bundles: \( M_1(1) \simeq \mathbb{T}^{2}_{A} \) with \( A = \begin{pmatrix}0 & 1 \\ -1 & 0\end{pmatrix} \), \( M_2(1) \simeq \mathbb{T}^{2}_{A} \) with \( A = \begin{pmatrix}-1 & 0 \\ 0 & -1\end{pmatrix} \), \( M'_1(1) \simeq \mathbb{T}^{2}_{A} \) with \( A = \begin{pmatrix}0 & -1 \\ 1 & 1\end{pmatrix} \), and \( M'_2(1) \simeq \mathbb{T}^{2}_{A} \) with \( A = \begin{pmatrix}-1 & -1 \\ 1 & 0\end{pmatrix} \). \noindent\text{Case 3: \( \chi(B) < 0 \).}
Here \( M \) has \( \mathbb{H}^2 \times \mathbb{R} \) geometry, and thus \( M \simeq (\mathbb{H}^2 \times \mathbb{R}) / \Gamma \), where \( \Gamma \subset \mathrm{Isom}(\mathbb{H}^2 \times \mathbb{R}) \) is a discrete subgroup acting freely \cite{scott1983geometries}. 

This completes the proof of the theorem.
\end{proof}

The following corollaries are direct consequences of Theorem~\ref{thm:Betti}:
\begin{cor}
Let $M$ be a closed, orientable 3-manifold admitting a nontrivial parallel vector
field $Z$ with respect to some Riemannian metric $g$. Let $\alpha$ be the $1$-form
metrically equivalent to $Z$, and denote by $\mathcal{L} = \ker(\alpha)$ the
corresponding foliation. Assume that the Ricci curvature of $g$ is non-negative
on $\mathcal{L}\times\mathcal{L}$ and that $\mathcal{L}$ contains a compact leaf. Then the first Betti
number $b_1(M)$ is either $1$ or $3$. Moreover, $b_1(M) = 3$ if and only if $M$ is,
up to a finite cover, diffeomorphic to the $3$-torus $\mathbb{T}^3$. If $b_1(M) = 1$ then the flow of $Z$ has either exactly two closed orbits or infinitely many closed orbits. In the first case $M$ is diffeomorphic to $\mathbb{S}^2 \times \mathbb{S}^1$, and in the second case $M$ is either a nontrivial $\mathbb{T}^{2}$-bundle over $\mathbb{S}^1$ or a quotient $(\mathbb{H}^2 \times \mathbb{R}) / \Gamma$, where $\Gamma \subset \mathrm{Isom}(\mathbb{H}^2 \times \mathbb{R})$ is a discrete subgroup acting freely.
\end{cor}
\begin{proof}
By assumption, the codimension-one foliation $\mathcal{L}$ contains a compact leaf $L$. 
Then, by \cite[Corollary~3.31]{tondeur1997geometry}, every leaf of $\mathcal{L}$ is diffeomorphic to $L$. 
Since the Ricci curvature of $g$ is non-negative on $\mathcal{L} \times \mathcal{L}$, 
it follows from \cite[Theorem~2.82]{larz2011global} that the first Betti number satisfies 
\[
1 \leq b_1(M) \leq 3.
\]
Combining this with Theorem~\ref{thm:Betti}, we deduce that $b_1(M)$ is either $1$ or $3$. 
Moreover, by \cite[Theorem~1.1]{schliebner2015lorentzian}, the equality $b_1(M) = 3$ holds if and only if 
$M$ is, up to a finite cover, diffeomorphic to the $3$-torus $\mathbb{T}^3$. Now assume that $b_1(M) = 1$. We distinguish two cases according to whether the number of closed orbits of $Z$ is finite or infinite. In the finite case, as noted in the comments following \cite[Corollary~8.7]{bazzoni2015k}, the number of closed Reeb orbits $N$ is given by $\dim H^{*}(M,\mathcal{F}_{Z})$, where $H^{*}(M,\mathcal{F}_{Z})$ is the basic cohomology of $M$ with respect to the foliation $\mathcal{F}_{Z}$ induced by $Z$. By Theorem~\ref{thm:Chern}, the manifold $M$ is K-cosymplectic. According to \cite[Theorem~7.6]{bazzoni2015k}, when $b_1(M) = 1$, the $\mathbb{S}^{1}$-action is automatically Hamiltonian (this fact is also noted on p.~31 of \cite{bazzoni2015k}). For a 3-dimensional K-cosymplectic manifold with $b_1(M) = 1$, \cite[Corollary~8.7]{bazzoni2015k} gives
\[
N = 1 + b_1(M) = 2.
\]
From \cite[Théorème~1]{carriere1984flots}, it follows that $M$ is diffeomorphic to $\mathbb{S}^2 \times \mathbb{S}^1$. The infinite case follows directly from Theorem~\ref{thm:Betti}.
\qedhere
\end{proof}

\begin{cor}
Let \((M,g)\) be a compact, orientable Lorentzian \(3\)-manifold. Then \(M\) has odd Betti numbers and is diffeomorphic to one of the manifolds listed in Theorem~\ref{thm:Betti} if either of the following holds:
\begin{enumerate}
    \item \(M\) admits a timelike vector field \(Z\) which is parallel with respect to \(g\); or
    \item \(M\) admits a closed \(1\)-form \(\alpha\) and a lightlike vector field \(Z\) with \(\alpha(Z) = \mathrm{const} \neq 0\), where \(Z\) is parallel with respect to \(g\).
\end{enumerate}

\end{cor}

\begin{proof}
\noindent\text{(1).} Suppose \(M\) admits a timelike vector field \(Z\) that is parallel with respect to the Lorentzian metric \(g\). Normalize so that \(g(Z,Z) = -1\), and let \(\beta = \iota_Z g\) be the associated 1-form. Define a Riemannian metric \(h\) by
\[
h = g + 2\beta \otimes \beta.
\]
Then \(h(Z,Z) = 1\), and since \(Z\) is parallel with respect to \(g\), the metrics \(g\) and \(h\) share the same Levi–Civita connection. Hence, \(Z\) is also parallel with respect to \(h\). The conclusion now follows from Theorem~\ref{thm:Betti}.

\noindent\text{(2).} Suppose \(M\) admits a closed 1-form \(\alpha\) and a lightlike vector field \(Z\) that is parallel with respect to the Lorentzian metric \(g\), with \(\alpha(Z) = \text{constant} \neq 0\). By \cite[Section 14]{zeghib1996killing}, \(Z\) is a Killing vector field with respect to some Riemannian metric. Theorem~\ref{thm:Chern} then implies that \(M\) admits a nontrivial parallel vector field with respect to some Riemannian metric, and the result follows from Theorem~\ref{thm:Betti}.
\end{proof}

 \bibliographystyle{splncs04}
\bibliography{main.bib}

@article{blair2010riemannian,
  title   = {Riemannian geometry of contact and symplectic manifolds},
  author  = {Blair, David E.},
  journal = {Progress in Mathematics},
  volume  = {203},
  year    = {2010},
  publisher = {Birkhäuser},
  address = {Boston}
}

@book{Wolf,
author = {Joseph A.Wolf},
title = {Space of constant curvature},
year = {1974},
publisher= {Boston Mass}
}

@article{marrero1998new,
  title={New examples of compact cosymplectic solvmanifolds},
  author={Marrero, JC and Padron, E},
  journal={Archivum Mathematicum},
  volume={34},
  number={3},
  pages={337--345},
  year={1998},
  publisher={Department of Mathematics, Faculty of Science of Masaryk University, Brno}
}

@article{li2008topology,
  title={Topology of co-symplectic/co-K{\"a}hler manifolds},
  author={Li, Hongjun},
  journal={Asian Journal of Mathematics},
  volume={12},
  number={4},
  pages={527--544},
  year={2008}
}

@article{hantzsche1935dreidimensionale,
  title={Dreidimensionale euklidische raumformen},
  author={Hantzsche, Walter and Wendt, Hilmar},
  journal={Mathematische Annalen},
  volume={110},
  number={1},
  pages={593--611},
  year={1935},
  publisher={Springer}
}

@article{bazzoni2015k,
  title={K-cosymplectic manifolds},
  author={Bazzoni, Giovanni and Goertsches, Oliver},
  journal={Annals of Global Analysis and Geometry},
  volume={47},
  number={3},
  pages={239--270},
  year={2015},
  publisher={Springer}
}

@article{goldberg1969integrability,
  title={Integrability of almost cosymplectic structures},
  author={Goldberg, Samuel and Yano, Kentaro},
  journal={Pacific Journal of Mathematics},
  volume={31},
  number={2},
  pages={373--382},
  year={1969},
  publisher={Mathematical Sciences Publishers}
}

@book{blair1966theory,
  title={The Theory of Quasi-Sasakian Structures.},
  author={Blair, David Ervin},
  year={1966},
  publisher={University of Illinois at Urbana-Champaign}
}

@article{bazzoni2014structure,
  title={On the structure of co-K{\"a}hler manifolds},
  author={Bazzoni, Giovanni and Oprea, John},
  journal={Geometriae Dedicata},
  volume={170},
  number={1},
  pages={71--85},
  year={2014},
  publisher={Springer}
}

@article{Chinea1993TopologyOC,
  title={Topology of cosymplectic manifolds},
  author={Domingo Chinea and Manuel de Le{\'o}n and Juan Carlos Marrero},
  journal={Journal de Math{\'e}matiques Pures et Appliqu{\'e}es},
  year={1993},
  volume={72},
  pages={567-591},
  url={https://api.semanticscholar.org/CorpusID:123755140}
}

@article{olszak1981almost,
  title={On almost cosymplectic manifolds},
  author={Olszak, Zbigniew},
  journal={Kodai Mathematical Journal},
  volume={4},
  number={2},
  pages={239--250},
  year={1981},
  publisher={Department of Mathematics, Tokyo Institute of Technology}
}

@article{bieberbach1911bewegungsgruppen,
  title={{\"U}ber die Bewegungsgruppen der euklidischen R{\"a}ume: erste Abhandlung},
  author={Bieberbach, Ludwig},
  journal={Mathematische Annalen},
  volume={70},
  number={3},
  pages={297--336},
  year={1911},
  publisher={Springer}
}

@article{carriere1984flots,
  title={Flots riemanniens},
  author={Carri{\`e}re, Yves},
  journal={Ast{\'e}risque},
  volume={116},
  pages={31--52},
  year={1984}
}

@article{scott1983geometries,
  title={The geometries of 3-manifolds},
  author={Scott, Peter},
  year={1983},
  publisher={Oxford University Press}
}

@article{martelli2016introduction,
  title={An introduction to geometric topology},
  author={Martelli, Bruno},
  journal={arXiv preprint arXiv:1610.02592},
  year={2016}
}

@article{bieberbach1912bewegungsgruppen,
  title={{\"U}ber die bewegungsgruppen der euklidischen r{\"a}ume (zweite abhandlung.) die gruppen mit einem endlichen fundamentalbereich},
  author={Bieberbach, Ludwig},
  journal={Mathematische Annalen},
  volume={72},
  number={3},
  pages={400--412},
  year={1912},
  publisher={Springer}
}

@article{chern1966geometry,
  title={The geometry of y-structures},
  author={Chern, Shiing-Shen},
  journal={Bulletin of the American Mathematical Society},
  volume={72},
  number={2},
  pages={167--219},
  year={1966}
}

@article{karp1977parallel,
  title={Parallel vector fields and the topology of manifolds},
  author={Karp, Leon},
  journal={Bulletin of the American Mathematical Society},
  volume={83},
  number={5},
  pages={1051-1053},
  year={1977}
}

@article{bott1967vector,
  title={Vector fields and characteristic numbers.},
  author={Bott, Raoul},
  journal={Michigan Mathematical Journal},
  volume={14},
  number={2},
  pages={231--244},
  year={1967},
  publisher={University of Michigan, Department of Mathematics}
}

@book{tondeur1997geometry,
  title={Geometry of foliations},
  author={Tondeur, Philippe},
  number={90},
  year={1997},
  publisher={Springer Science \& Business Media}
}

@article{percell1981parallel,
  title={Parallel vector fields on manifolds with boundary},
  author={Percell, Peter},
  journal={Journal of Differential Geometry},
  volume={16},
  number={1},
  pages={101--104},
  year={1981},
  publisher={Lehigh University}
}

@article{welsh1986manifolds,
  title={Manifolds that admit parallel vector fields},
  author={Welsh Jr, David J},
  journal={Illinois Journal of Mathematics},
  volume={30},
  number={1},
  pages={9--18},
  year={1986},
  publisher={Duke University Press}
}

@article{cappelletti2013survey,
  title={A survey on cosymplectic geometry},
  author={Cappelletti-Montano, Beniamino and De Nicola, Antonio and Yudin, Ivan},
  journal={Reviews in Mathematical Physics},
  volume={25},
  number={10},
  pages={1343002},
  year={2013},
  publisher={World Scientific}
}

@article{larz2011global,
  title={Global aspects of holonomy in pseudo-Riemannian geometry},
  author={L{\"a}rz, Kordian},
  year={2011},
  pages={ Ph.D. thesis},
  journal={Humboldt-Universit{\"a}t zu Berlin, Mathematisch-Naturwissenschaftliche Fakult{\"a}t II}
  
}

@inproceedings{schliebner2015lorentzian,
  title={On Lorentzian manifolds with highest first Betti number},
  author={Schliebner, Daniel},
  booktitle={Annales de l'Institut Fourier},
  volume={65},
  number={4},
  pages={1423--1436},
  year={2015}
}

@article{zeghib1996killing,
  title={Killing fields in compact Lorentz $\{$3$\}$-manifolds},
  author={Zeghib, Abdelghani},
  journal={Journal of Differential Geometry},
  volume={43},
  number={4},
  pages={859--894},
  year={1996},
  publisher={Lehigh University}
}
\end{document}